\documentclass{article}

\usepackage{amsmath,amscd,amssymb,amsthm,enumerate}
\usepackage[matrix,arrow,curve]{xy}

\newtheorem{thm}{Theorem}[section]
\newtheorem{lemma}[thm]{Lemma}
\newtheorem{proposition}[thm]{Proposition}

\newtheorem{claim}[thm]{Claim}

\theoremstyle{definition}

\newtheorem{remark}[thm]{Remark}

\newcommand{\pr}{\mathbb{P}}
\newcommand{\N}{\mathbb{N}}
\newcommand{\Z}{\mathbb{Z}}

\newcommand{\R}{\mathbb{R}}
\newcommand{\C}{\mathbb{C}}
\newcommand{\F}{\mathbb{F}}
\newcommand{\NC}{\operatorname{N}_1}

\newcommand{\NE}{\operatorname{NE}}

\newcommand{\Bl}{\operatorname{Bl}}
\newcommand{\Br}{\operatorname{Br}}
\newcommand{\Exc}{\operatorname{Exc}}

\newcommand{\Pic}{\operatorname{Pic}}
\newcommand{\Ext}{\operatorname{Ext}}

\newcommand{\sm}{\operatorname{sm}}
\newcommand{\cont}{\operatorname{cont}}

\newcommand{\sO}{\mathcal{O}}
\newcommand{\sN}{\mathcal{N}}
\newcommand{\sE}{\mathcal{E}}

\title{Fano manifolds having $(n-1,0)$-type extremal rays with large Picard number}
\author{Kento Fujita}

\begin{document}
\maketitle
\begin{abstract}
{\noindent We classify smooth Fano manifolds $X$ 
with the Picard number $\rho_X\geq 3$ such that there exists an extremal ray 
which has a birational contraction that maps a divisor to a point.}
\end{abstract}

\section{Introduction}
Let $X$ be a (smooth) Fano manifold with $\dim X\geq 3$. 
It is very powerful tool to classify and evaluate those $X$ that we see 
\emph{extremal rays} of $X$. In fact, Mori and Mukai succeeded in classifying 
Fano threefolds by viewing extremal rays in detail \cite{MoMu}. 
It is very difficult to consider higher dimensional Fano manifolds in general. 
However, if there exists a ``special" extremal ray, then we can get 
various information about $X$. 
For example, Bonavero, Campana and Wi\'sniewski classified in \cite{bcw} 
that the Fano manifold $X$ 
which has a extremal ray which induces the blowing up of a smooth variety 
along a point. 

Recently, Tsukioka and Casagrande (see \cite{tsukioka,casagrande}) showed that if 
there exists an extremal ray which has a birational contraction 
that maps a divisor to a point, then the Picard number $\rho_X$ of $X$ is at most three. 
This can be seen a kind of generalization of the result of \cite{bcw}. 

Our main result is to classify those $X$ with maximal Picard number. 

\begin{thm}[Main Theorem]\label{mainthm}
Let $X$ be a smooth projective variety of dimension $n\geq 3$. 
Then the following are equivalent:
\begin{enumerate}
\renewcommand{\theenumi}{\arabic{enumi}}
\renewcommand{\labelenumi}{$(\theenumi)$}
\item\label{mainthm1}
$X$ is a Fano manifold such that $\rho_X\geq 3$ and there exists an extremal ray $R\subset\NE(X)$ of type $(n-1,0)$,
\item\label{mainthm2}
$X\simeq\Bl_WY$ such that the following holds:
\begin{enumerate}
\renewcommand{\theenumii}{\alph{enumii}}
\renewcommand{\labelenumii}{\rm{(\theenumii)}}
\item\label{mainthm2a}
$\pi\colon Y=\pr_Z(\sO_Z\oplus\sO_Z(s))\rightarrow Z,$
where $Z$ is an $(n-1)$-dimensional Fano manifold of $\rho_Z=1$ 
with index $r$ such that 
the ample generator of $\Pic(Z)$ is $\sO_Z(1)$. 
\item\label{mainthm2b}
The inequality $r>s>0$ holds.
\item\label{mainthm2c}
$D\subset Y$ is a section of $\pi$ with $\sN_{D/Y}\simeq\sO_Z(s)$ 
and $W\subset D$ is a smooth divisor with $W\in|\sO_Z(d)|$.
\item\label{mainthm2d}
The inequality $r>d-s$ holds. 
\end{enumerate}
\end{enumerate}
\end{thm}

\medskip

\noindent\textbf{Acknowledgements.}
The author would like to express his graditude to Professor Shigefumi Mori 
for warm encouragements.
The author is partially supported by JSPS Fellowships for Young Scientists.

\medskip

\noindent\textbf{Notation and terminology.}
We always work over the complex number field $\C$.
The theory of extremal contraction, we refer the readers to \cite{KoMo}. 
For a smooth projective variety $X$ of dimension $n$ 
and a $K_X$-negative extremal ray 
$R\subset\overline{\NE}(X)$, let $\cont_R\colon X\rightarrow Y_R$ be the associated 
extremal contraction corresponds to $R$. We also let 
\[
\Exc(R):=\Exc(\cont_R)=\{x\in X|\cont_R \text{ is not isomorphism around }x\}.
\]
We say $R$ is \emph{of fiber type} (resp.\ \emph{divisorial, small}) if 
the associated contraction morphism 
${\cont}_R\colon X\rightarrow Y$ is of fiber type (resp.\ divisorial, small).
We say that $R$ (or $\cont_R$) is \emph{of type} $(m,l)$ (or \emph{of $(m,l)$-type}) 
if $\dim\Exc(R)=m$ and $\dim\cont_R(\Exc(R))=l$. 
We also say that $R$ (or $\cont_R$) is \emph{of type} $(n-1,n-2)^{\sm}$ 
(or \emph{of $(n-1,n-2)^{\sm}$-type}) 
if the morphism $\cont_R$ is the blowing up morphism of 
a smooth projective variety along a smooth subvariety of codimension $2$. 

For a proper variety $X$, the Picard number of $X$ is denoted by $\rho_X$. 
For a closed subvariety $Y\subset X$, let $\NC(Y,X)$ be the image of the morphism $\NC(Y)\rightarrow\NC(X)$. 

We say $X$ is a \emph{Fano manifold} if $X$ is a smooth projective variety 
whose anticanonical divisor $-K_X$ is ample. 
For a Fano manifold $X$, let its \emph{index} be 
\[
\max\{m\in\mathbb{N}\mid-K_X\sim mL\text{ for some Cartier divisor }L\}.
\]

For abbreviation, we let \emph{pt} stand \emph{point}.

\medskip

\section{Preliminaries}

\smallskip

We consider the case that there exists a prime divisor $E\subset X$ such that $\dim\NC(E,X)=1$. 

\begin{lemma}\label{n1}
Let $X$ be an $n$-dimensional smooth projective variety and 
$E\subset X$ be a prime divisor with $\dim\NC(E,X)=1$. 
We assume that there exists a $K_X$-negative extremal ray 
$R\subset\overline{\NE}(X)$ such that $E\cap\Exc(R)\neq\emptyset$ and 
$\cont_R(E)$ is not a point. 
Then we have $(E\cdot R)>0$ and 
$R$ is either of type $(n-1,n-2)^{\sm}$ or the morphism $\cont_R$ is 
a conic bundle. 
\end{lemma}

\begin{proof}
There exists an irreducible curve $C\subset X$ with $C\cap E\neq\emptyset$ and $[C]\in R$. 
If $C\subset E$, then $\cont_R(E)=\text{pt}$ holds. This leads to a contradiction. 
Hence $(E\cdot R)>0$ holds 
since $C\cap E\neq\emptyset$ and $C\not\subset E$. 
Using same argument, we have $C'\cap E\neq\emptyset$ and $C'\not\subset E$ 
for any curve $C'\subset X$ with $[C']\in R$ since $(E\cdot C')>0$ holds. 
If there exists a closed subvariety $S\subset X$ such that $\cont_R(S)=\text{pt}$ and $\dim S\geq 2$, then 
there exists an irreducible curve $C\subset S\cap E$. 
However, such $C$ satisfies that $[C]\in R$ and $C\subset E$, 
this leads to a contradiction. 
Thus all fiber of $\cont_R$ are of dimension $\leq 1$. 
Therefore $R$ is of type $(n-1,n-2)^{\sm}$ or $\cont_R$ is a conic bundle by 
Ando's classification result \cite{ando,wisn}. 
\end{proof}

We classify smooth projective varieties having $(n-1,0)$-type or $(n,1)$-type 
($K_X$-negative) extremal contraction 
and having $\pr^1$-bundle structure. This classification result is essential 
for the proof of Theorem \ref{mainthm}. 

\begin{proposition}\label{p1}
Let $Y$ be an $n$-dimensional smooth projective variety. 
We assume that there exists distinct $K_X$-negative extremal rays 
$R$, $R'\subset\overline{\NE}(Y)$ with the associated contraction morphisms 
$\sigma:=\cont_R\colon Y\rightarrow V$, $\pi:=\cont_{R'}\colon Y\rightarrow Z$, respectively. We assume that those $R$ and $R'$ satisfy the following properties:
\begin{enumerate}
\renewcommand{\theenumi}{\arabic{enumi}}
\renewcommand{\labelenumi}{$(\theenumi)$}
\item\label{p101}
There exists a prime divisor $E\subset Y$ such that $\sigma(E)=\text{pt}$.
\item\label{p102}
$\pi$ is a $\pr^1$-bundle. 
\end{enumerate}
Then $Y\simeq\pr_Z(\sO_Z\oplus\sO_Z(s))\xrightarrow{\pi}Z$ such that
\begin{enumerate}
\renewcommand{\theenumi}{\arabic{enumi}}
\renewcommand{\labelenumi}{$(\theenumi)$}
\item\label{p111}
$Z$ is an $(n-1)$-dimensional Fano manifold of $\rho_Z=1$ with index $r$ 
such that the ample generator of $\Pic(Z)$ is $\sO_Z(1)$,
\item\label{p112}
$r>s\geq 0$. 
\end{enumerate}
\end{proposition}

\begin{proof}
We can show that $R$ is of type $(n-1,0)$ or $(n,1)$ since $\sigma(E)=\text{pt}$. 
We replace $E$ by a general smooth fiber of $\sigma$ if $R$ is of type $(n,1)$. 
The restriction morphism $\pi|_E\colon E\rightarrow Z$ is a finite morphism, 
thus $\pi|_E$ is surjective. 
We have $\rho_Z=1$ since $\dim\NC(E,Y)=1$. Hence $\rho_Y=2$ holds. 
In particular, $Y$ is a Fano manifold since there exists at least two $K_X$-negative extremal rays. 
Therefore $Z$ is also a Fano manifold by \cite[Corollary 2.9]{KMM}. 
Let $r$ be the index of $Z$ and $\sO_Z(1)$ be the 
ample generator of $\Pic(Z)$. 

\begin{claim}\label{sect}
$E$ is a section of $\pi$ $($i.e. the restriction morphism 
$\pi|_E\colon E\rightarrow Z$ is isomorphism$)$. 
\end{claim}

\begin{proof}[{Proof of Claim \ref{sect}}]
If $\pi|_E$ is unramified, then $\pi|_E$ is 
\'etale , hence $\pi|_E$ is isomorphism since $E$ and $Z$ are smooth Fano manifolds 
under the assumption. 
Thus it is enough to show that $\pi|_E$ is unramified. 

We assume that there exists a branch point $z\in Z$ of $\pi|_E$. 
Then we can pick general smooth (very free) rational curve $z\in B\subset Z$ with $B\not\subset\Br(\pi|_E)$, 
where $\Br(\pi|_E)$ is the branch locus of $\pi|_E$. 
We note that $Z$ is a rationally connected variety since $Z$ is a Fano manifold 
(see \cite{KMM}). 
Then the morphism $S:=\pi^{-1}(B)\rightarrow B$ is isomorphic to a Hirzebruch surface with the ruling $\F_m=\pr_{\pr^1}(\sO_{\pr^1}\oplus\sO_{\pr^1}(m))\rightarrow\pr^1$ 
for some $m\geq 0$. 
We note that $E\cap S\subset S$ is a reduced divisor and $\sigma(E\cap S)=\text{pt}$. 
Thus the Stein factorization $\sigma'\colon S\rightarrow T$ of the morphism 
$\sigma|_S\colon S\rightarrow V$ satisfies either of the following:
\begin{enumerate}
\item
$m=0$ and $\sigma'$ is a projection onto $\pr^1$.
\item
$m>0$ and $\sigma'$ is the contraction morphism contracting the $(-m)$-curve. 
\end{enumerate}
Therefore $E\cap S\subset S$ is the sum of disjoint union of sections of $\pi|_S$ in any case. 
However, this contradict to the choice of $z\in Z$ since 
$z\in Z$ is a branch point of $\pi|_E$. 
\end{proof}

Thus we can write $Y=\pr_Z(\sE)$ with $\sE:=\pi_*\sO_Y(E)$. 
We take $s\in\Z$ such that $\sO_E(E)\simeq\sO_Z(-s)$. 
We have $s\geq 0$; if $R$ is of type $(n-1,0)$ then $s>0$ since $\sO_E(E)$ 
is anti-ample, 
if $R$ is of type $(n,1)$ then $s=0$ since $E$ is a general smooth fiber of $\sigma$. 

We consider the exact sequence 
\[
0\rightarrow\sO_Y\rightarrow\sO_Y(E)\rightarrow\sO_E(E)\rightarrow 0.
\]
We obtain 
\[
0\rightarrow\sO_Z\rightarrow\sE\rightarrow\sO_Z(-s)\rightarrow 0.
\]
We have $\sE\simeq\sO_Z\oplus\sO_Z(-s)$ since $\Ext^1(\sO_Z(-s),\sO_Z)=0$. 
Consequently, $r>s(\geq 0)$ holds 
since $\sO_E(-K_Y|_E)\simeq\sO_E(-K_E+E|_E)\simeq\sO_Z(r-s)$ is ample. 
\end{proof}

\begin{remark}\label{rhotwormk}
The ray $R$ in Proposition \ref{p1} is 
\begin{enumerate}
\renewcommand{\theenumi}{\alph{enumi}}
\renewcommand{\labelenumi}{(\theenumi)}
\item\label{rhotwormk1}
of type $(n-1, 0)$ if and only if $s>0$, 
\item\label{rhotwormk2}
of type $(n, 1)$ if and only if $s=0$ (hence $Y\simeq Z\times\pr^1$). 
\end{enumerate}
We also remark that the case \eqref{rhotwormk1} is exactly the case in 
\cite[Proposition 2.5]{fuj}. 
\end{remark}

Now, we prove the easy direction of Theorem \ref{mainthm}. 

\begin{lemma}\label{conv}
Let $X$ be an $n$-dimensional smooth projective variety with $n\geq 3$. 
We assume that $X=\Bl_WY$ and $X$ satisfies the conditions of 
\eqref{mainthm2a}, \eqref{mainthm2b} and \eqref{mainthm2c} in Theorem \ref{mainthm}. 
Then $X$ is a Fano manifold if and only if $r>d-s$. 
\end{lemma}

\begin{proof}
Let $e_0\subset Z$ be a general irreducible curve, $m$ be the intersection number 
$(\sO_Z(1)\cdot e_0)$, $E\subset Y$ be the section of $\pi$ 
with $\sN_{E/Y}\simeq\sO_Z(-s)$ 
and $E\subset X$ be its strict=total transform 
(same notation but there are no confusion). 
Let $D'\subset X$ be the strict transform of $D\subset Y$, $e\subset E$ and $e'\subset D'$ 
be the strict transform of the curve $e_0\subset Z$. Finally, let $F\subset X$ be 
the exceptional divisor of the blowing up $\phi\colon X\rightarrow Y$, 
$f\subset X$ be a nontrivial fiber of $\phi$ 
and $f'\subset X$ be the strict transform of a fiber of $\pi$ passing through 
$W\subset Y$. 
Then we can show the following:

\begin{claim}\label{ray}
\begin{enumerate}
\renewcommand{\theenumi}{\arabic{enumi}}
\renewcommand{\labelenumi}{$(\theenumi)$}
\item
$\sO_X(-K_X)\simeq\sO_X(2E-F)\otimes(\pi\circ\phi)^*\sO_Z(r+s)$.
\item
$\NE(X)=\R_{\geq 0}[e]+\R_{\geq 0}[e']+\R_{\geq 0}[f]+\R_{\geq 0}[f']$.
\item
For an irreducuble curve $C\subset X$, $C\subset E$ holds if and only if $[C]\in\R_{\geq 0}[e]$ holds. 
\item
We obtain the following table of intersection numbers: 
\begin{center}
\begin{tabular}[t]{|c|c|c|c|c|} \hline
 & $E$ & $(\pi\phi)^*\sO_Z(1)$ & $F$ & $-K_X$ \\ \hline
$e$ & $-sm$ & $m$ & $0$ & $(r-s)m$ \\ \hline
$e'$ & $0$ & $m$ & $dm$ & $(r+s-d)m$ \\ \hline
$f$ & $0$ & $0$ & $-1$ & $1$ \\ \hline
$f'$ & $1$ & $0$ & $1$ & $1$ \\ \hline
\end{tabular}
\end{center}
\end{enumerate}
\end{claim}
Therefore $X$ is a Fano manifold if and only if $r>d-s$ by Claim \ref{ray}. 
\end{proof}

\begin{remark}\label{number}
The cone $\NE(X)$ is spanned by four rays if and only if $d-s>0$, 
three rays if and only if $d-s\leq 0$ by Claim \ref{ray}. 
\end{remark}

\medskip

\section{Proof of Theorem \ref{mainthm}}

\smallskip

Let $X$ be an $n$-dimensional Fano manifold with $\rho_X\geq 3$ 
and there exists an extremal ray $R\subset\NE(X)$ of type $(n-1,0)$. 
We can assume $\rho_X=3$ by 
\cite[Proposition 5]{tsukioka} and \cite[Proposition 3.1]{casagrande}. 
Let $E:=\Exc(R)$. We note that $\dim\NC(E,X)=1$ holds. 
We start to prove Theorem \ref{mainthm} by seeing the cone $\NE(X)$ in detail. 

\begin{proposition}\label{step1}
For any extremal ray $R_0\subset\NE(X)$ different from $R$, the ray 
$R_0$ is of birational type. 
Furthermore, we have the following properties: 
\begin{enumerate}
\renewcommand{\theenumi}{\arabic{enumi}}
\renewcommand{\labelenumi}{$(\theenumi)$}
\item\label{step101}
If $E\cap\Exc(R_0)\neq\emptyset$, then we have 
\begin{enumerate}
\renewcommand{\theenumii}{\alph{enumii}}
\renewcommand{\labelenumii}{\rm{(\theenumii)}}
\item\label{step101a}
$(E\cdot R_0)>0,$
\item\label{step101b}
$R_0$ is of type $(n-1,n-2)^{\sm}$ $($let $F_0\subset X$ be its exceptional divisor$)$,
\item\label{step101c}
$(F_0\cdot R)>0$. 
\end{enumerate}
\item\label{step102}
If $E\cap\Exc(R_0)=\emptyset$, then $R+R_0\subset\NE(X)$ is an extremal face 
of $\NE(X)$. 
\end{enumerate}
\end{proposition}

\begin{proof}
Let $R_0\subset\NE(X)$ be an arbitrally extremal ray different from $R$. 

We assume that $E\cap\Exc(R_0)\neq\emptyset$. It is obvious that 
$\cont_{R_0}(E)\neq\text{pt}$. 
Hence \eqref{step101} holds by Lemma \ref{n1}. 
We note that if the morphism $\cont_{R_0}\colon X\rightarrow Y_0$ is a conic bundle, 
then $\rho_{Y_0}=1$ hence $\rho_X=2$ 
since $\cont_{R_0}(E)=Y_0$ and $\dim\NC(E,X)=1$. 

We assume that $E\cap\Exc(R_0)=\emptyset$. 
It is obvious that $R_0$ is of birational type. 
If $R+R_0$ does not span an extremal face, 
then there exists an extremal ray $R'\subset\NE(X)$ different from $R$ such that $(E\cdot R')<0$ since $(E\cdot R)<0$ and $(E\cdot R_0)=0$. 
However, applying \eqref{step101} for $R'$, this leads to a contradiction. 

In particular, $R_0$ is of birational type in any case. 
\end{proof}

The second step of the proof is to consider extremal rays of type $(n-1,n-2)^{\sm}$, 
in any case the image of the contraction morphism is again a Fano manifold and has a $\pr^1$-bundle structure. 
Furthermore, we can see that there is an elementally transform factoring $X$. 

\begin{proposition}\label{step2}
For any extremal ray $R_0\subset\NE(X)$ of type $(n-1,n-2)^{\sm}$, 
let $\phi_0:=\cont_{R_0}\colon X\rightarrow Y_0$, 
let $F_0\subset X$ be the exceptional divisor of $\phi_0$ 
and let $W_0\subset Y_0$ be the $($smooth$)$ center of the blowing up $\phi_0$. 
Then $Y_0$ is a Fano manifold which has a $\pr^1$-bundle structure 
$\pi\colon Y_0\rightarrow Z$. 
Furthermore, there exists an elementally transform factoring $X$. 
More precisely, there exists a commutative diagram 
\[\xymatrix{
& X \ar[dl]_{\phi'} \ar[dr]^{\phi_0} & \\
Y'\ar[dr]_{\pi'} & & Y_0 \ar[dl]^{\pi} \\
& Z & \\
}\]
such that $Y'$ is a smooth projective variety, 
$\phi'$ is the blowing up $($different from $\phi_0$$)$ 
along smooth subvariety of codimension $2$ whose exceptional divisor id the strict 
transform of the divisor $\pi^{-1}(\pi(W_0))$
and $\pi'$ is a $\pr^1$-bundle. 
\end{proposition}

\begin{proof}
Let $E_0:=\phi_0(E)$. 
We prove Proposition \ref{step2} by dividing these two cases: 
\begin{enumerate}
\renewcommand{\theenumi}{\Alph{enumi}}
\renewcommand{\labelenumi}{\rm{(\theenumi)}}
\item\label{caseA}
$E\cap\Exc(R_0)\neq\emptyset$,
\item\label{caseB}
$E\cap\Exc(R_0)=\emptyset$. 
\end{enumerate}

\begin{claim}\label{step21}
$Y_0$ is a Fano manifold. 
\end{claim}

\begin{proof}[Proof of Claim \ref{step21}]
We consider the case \eqref{caseA}. 
It is enough to show $(-K_{Y_0}\cdot C)>0$ for all irreducible curves in $W_0$ 
since $\overline{\NE}(Y_0)=\NE(Y_0)$. 
We note that $\dim\NC(E_0,Y_0)=1$. Thus all curves in $E_0$ are 
numerically proportional. Hence $(-K_{Y_0}\cdot C)>0$ holds. 

We consider the case \eqref{caseB}. 
We know that $R+R_0\subset\NE(X)$ is an extremal face 
by Proposition \ref{step1} \eqref{step102}. 
Let $R_1\subset\NE(X)$ be the unique extremal ray 
different from $R$ and $R_0$ such that $R_0+R_1\subset\NE(X)$ spans an extremal face. 
We have $(E\cdot R_1)>0$ since $(E\cdot R_0)=0$ and $(E\cdot R)<0$. 
Thus $R_1$ is of type $(n-1,n-2)^{\sm}$ by Proposition \ref{step1} \eqref{step101}. 
Let $\phi_1:=\cont_{R_1}\colon X\rightarrow Y_1$, 
$F_1\subset X$ be the exceptional divisor of $\phi_1$ 
and $W_1\subset Y_1$ be the (smooth) center of the blowing up $\phi_1$. 
We have $F_0\neq F_1$ since $E\cap F_0=\emptyset$ and $E\cap F_1\neq\emptyset$. 
Thus there exists a nontrivial fiber $C_1\subset X$ of $\phi_1$ 
such that $C_1\not\subset F_0$. 
We consider the surjective map ${\phi_0}_*:\NE(X)\twoheadrightarrow\NE(Y_0)$. 
Let $R_{Y_0}$, ${R_1}_{Y_0}\subset\NE(Y_0)$ be the images of $R$, $R_1\subset\NE(X)$, 
respectively. 
Then the cone $\NE(Y_0)$ is spanned by $R_{Y_0}, {R_1}_{Y_0}$. 
We note that $R_{Y_0}$ is a $K_X$-negative extremal ray of type $(n-1,0)$ 
since $(E_0\cdot R_{Y_0})<0$. 
Since ${R_1}_{Y_0}$ is spanned by the class $[{\phi_0}_*C_1]$ and 
\[
(-K_{Y_0}\cdot {\phi_0}_*C_1)=(-K_X\cdot C_1)+(F_0\cdot C_1)>0,
\]
both rays are $K_X$-negative. Hence $Y_0$ is a Fano manifold. 
\end{proof}

There exists a $K_X$-negative extremal ray $R^0\subset\NE(Y_0)$ such that 
$(E_0\cdot R^0)>0$ by Claim \ref{step21}. 
Let $\pi:=\cont_{R^0}\colon Y_0\rightarrow Z$. By Lemma \ref{n1}, 
the morphism $\pi$ is of type $(n-1,n-2)^{\sm}$ or a conic bundle. 

\begin{claim}\label{step22}
$\pi|_{W_0}:W_0\rightarrow Z$ is a finite morphism. 
\end{claim}

\begin{proof}[Proof of Claim \ref{step22}]
We consider the case \eqref{caseA}. 
If there exists an irreducible curve $C\subset W_0$ such that $\pi(C)=\text{pt}$, then 
$\pi(E_0)=\text{pt}$ holds 
since $W_0\subset E_0$ and $\dim\NC(E_0,Y_0)=1$. This leads to a contradiction. 

We consider the case \eqref{caseB}. 
If there exists an irreducible curve $C\subset W_0$ such that $\pi(C)=\{z\}$, 
then $\dim\phi_0^{-1}\pi^{-1}(z)=2$. 
Hence there exists an irreducible curve $C'\subset E\cap\phi_0^{-1}\pi^{-1}(z)$, 
however we have 
$[C']\in R\cap(R_0+R_1)$. Hence this leads to a contradiction. 
\end{proof}

\begin{claim}\label{step23}
There exists an irreducible curve $f\subset Y_0$ with $[f]\in R^0$ such that $f\cap W_0\neq\emptyset$. 
\end{claim}

\begin{proof}[Proof of Claim \ref{step23}]
We assume the contrary. We can assume that $\pi$ is of type $(n-1,n-2)^{\sm}$. 
Let $G\subset Y_0$ be the exceptional divisor of $\pi$. 

We consider the case \eqref{caseA}. 
We have $G\cap W_0=\emptyset$ by assumption. 
We note that $E_0\cap G\neq\emptyset$ and $E_0\neq G$. 
Thus there exists an irreducible curve $C\subset E_0$ such that $(G\cdot C)>0$ holds. Since $\dim\NC(E_0,Y_0)=1$, 
$(G\cdot C)>0$ holds for any irreducible curve $C\subset E_0$. 
In particular, this holds for any curve in $W_0$ intersects $G$. 
This contradict to the property $G\cap W_0=\emptyset$. 

We consider the case \eqref{caseB}. 
We have $\pi(E_0)\cap\pi(W_0)=\emptyset$ 
since $E_0\cap W_0=\emptyset$ and $G\cap W_0=\emptyset$. 
We note that $\pi(E_0)\subset Z$ is a divisor and $\pi(W_0)$ contains a curve. 
This leads to a contradiction since $\rho_Z=1$. 
\end{proof}

For any irreducible curve $f\subset Y_0$ with $[f]\in R^0$ such that 
$f\cap W_0\neq\emptyset$ (we note that such $f$ always exists by Claim \ref{step23}), 
we can pick the strict transform $\hat{f}\subset X$ by Claim \ref{step22}. 
Then we have 
\[
0<(-K_X\cdot\hat{f})=(-K_{Y_0}\cdot f)-(F_0\cdot\hat{f}).
\]
Since $\pi$ is of type $(n-1,n-2)^{\sm}$ or a conic bundle, we have 
$(-K_{Y_0}\cdot f)=1$ or $2$. 
Moreover, by the choice of $f$, 
we have $(F_0\cdot \hat{f})\geq 1$. Hence $(-K_{Y_0}\cdot f)=2$ and $(F_0\cdot\hat{f})=1$ holds. 
Therefore we can show that $\pi$ is a conic bundle and $\pr^1$-bundle around $f$, $\deg(\pi|_{W_0})=1$ 
and $\Delta_{\pi}\cap\pi(W_0)=\emptyset$, where $\Delta_{\pi}$ be the 
discriminant divisor of the conic bundle $\pi$ (see for example \cite[\S 4]{wisn}). 
We have $\Delta_{\pi}=\emptyset$ since $\rho_Z=1$. 
Therefore $\pi$ is a $\pr^1$-bundle. Thus 
there exists an elementally transform passing through $X$. More presicely, $X$ has a contraction morphism $\phi':X\rightarrow Y'$, 
the exceptional divisor is the strict transform of $\pi^{-1}\pi(W_0)$, 
and $Y'$ has a $\pr^1$-bundle structure $\pi'\colon Y'\rightarrow Z$ such that $\pi\circ\phi_0=\pi'\circ\phi'$. 
\end{proof}

\begin{remark}\label{rmkcase1}
There exists an extremal ray $R_0\subset\NE(X)$ such that $(E\cdot R_0)>0$ holds 
since $X$ is a Fano manifold. 
Then $R_0$ is of type $(n-1,n-2)^{\sm}$ by Lemma \ref{n1}. 
Thus $R_0$ satisfies the assumptions of Proposition \ref{step2}. 
In particular, there exists at least two rays in $\NE(X)$ of type $(n-1,n-2)^{\sm}$ 
($R_0$ and its elementally transform). 
\end{remark}

\begin{remark}\label{rmkcase2}
If there exists an extremal ray $R_0\subset\NE(X)$ such that $E\cap\Exc(R_0)=\emptyset$ holds, then 
$Y_0$ has two extremal contractions such that:
\begin{enumerate}
\renewcommand{\theenumi}{\arabic{enumi}}
\renewcommand{\labelenumi}{(\theenumi)}
\item\label{rmkcase21}
$(n-1,0)$-type that maps $E_0$ to a point, and
\item\label{rmkcase22}
$\pr^1$-bundle structure
\end{enumerate}
by the proof of Proposition \ref{step2}, where let $\cont_{R_0}\colon X\rightarrow Y_0$. 
Hence $Y_0\simeq\pr_Z(\sO_Z\oplus\sO_Z(s))\xrightarrow{\pi}Z$ such that 
$Z$ is a Fano manifold of $\rho_Z=1$ with index $r$ such that 
the ample generator of $\Pic(Z)$ is $\sO_Z(1)$ and $r>s>0$ by Proposition \ref{p1}. 
Since $W_0\cap E_0=\emptyset$ and $\deg\pi|_{W_0}=1$ holds, there exists $D\subset Y_0$, the unique section of $\pi$, such that 
$E_0\cap D=\emptyset$ and $W_0\subset D$ holds. Since $W_0\subset D$ is a smooth divisor, there exists $d\in\N$ 
such that $W_0\in|\sO_Z(d)|$. We note that $r>d-s$ holds by Lemma \ref{conv}. 

Therefore, to prove Theorem \ref{mainthm}, it is eonugh to show that there exists 
an extremal ray $R_0\subset\NE(X)$ of type $(n-1,n-2)^{\sm}$ with $(E\cdot R_0)=0$. 
\end{remark}

We consider two cases whether there exists another ``special" extremal ray different from $R$ or not.

\begin{proposition}\label{step31}
If there exists an extremal ray $R'\subset\NE(X)$ different from $R$ such that $R'$ is not of type $(n-1,n-2)^{\sm}$. 
Then there exists an extremal ray $R_0\subset\NE(X)$ of type $(n-1,n-2)^{\sm}$ with $(E\cdot R_0)=0$. 
\end{proposition}

\begin{proof}
There exists a closed subvariety $E'\subset X$ with $\dim E'\geq 2$ such that $\cont_{R'}(E')=\text{pt}$ by Proposition \ref{step1}. 
We note that $E\cap\Exc(R')=\emptyset$ and 
$R+R'\subset\NE(X)$ spans an extremal face by Proposition \ref{step1}. 
We denote 
\[
\NE(X)=R+R'+R_1+\cdots+R_m,
\]
where the set of $2$-dimensional extremal faces of $\NE(X)$ is $\{R+R', R'+R_1, R_1+R_2,\cdots, R_{m-1}+R_m, R_m+R\}$. 
We have $m\geq 2$ by Remark \ref{rmkcase1}.

We assume that $(E\cdot R_i)>0$ for some $1\leq i\leq m$. 
Then $R_i$ is of type $(n-1,n-2)^{\sm}$ by Proposition \ref{step1}. 
Let $\phi_i:=\cont_{R_i}\colon X\rightarrow Y_i$, 
let $F_i\subset X$ be the exceptional divisor of $\phi_i$ 
and let $W_i\subset Y_i$ be the (smooth) center of the blowing up $\phi_i$. 

\begin{claim}\label{step311}
$(F_i\cdot R')=0$ holds. 
\end{claim}

\begin{proof}[Proof of Claim \ref{step311}]
If $(F_i\cdot R')\neq 0$, then $F_i\cap E'\neq\emptyset$. Thus there exists an irreducible curve $C\subset F_i\cap E'$. 
We consider the surjective map ${\phi_i}_*:\NE(X)\twoheadrightarrow\NE(Y_i)$. 
Let $R_{Y_i}$, $R'_{Y_i}\subset\NE(Y_i)$ be the images of $R$, $R'\subset\NE(X)$, respectively. 
It is obvious that $R_{Y_i}\cap R'_{Y_i}=\{0\}\subset\NE(Y_i)$. 
We have $[{\phi_i}_*(C)]\in R'_{Y_i}$ since $C\subset E'$. 
We also have $[{\phi_i}_*(C)]\in R_{Y_i}$ since $\phi_i(C)\subset W_i\subset\phi_i(E)$ 
and $\dim\NC(\phi_i(E),Y_i)=1$. 
However, if $\phi_i(C)=\text{pt}$ then we have $[C]\in R_i\cap R'=\{0\}$. 
Thus $\phi_i(C)\neq\text{pt}$, this leads to a contradiction. 
\end{proof}

\begin{claim}\label{step312}
For any $2\leq i\leq m$, we have $(E\cdot R_i)=0$. 
\end{claim}

\begin{proof}[Proof of Claim \ref{step312}]
It is enough to show that $(E\cdot R_i)\leq 0$ for any $2\leq i\leq m$ 
by Proposition \ref{step1}. 

We assume that there exists $2\leq i\leq m$ such that $(E\cdot R_i)>0$. 
Then $R_i$ is of type $(n-1,n-2)^{\sm}$ by Proposition \ref{step1}. 
Let $\phi_i:=\cont_{R_i}\colon X\rightarrow Y_i$, 
let $F_i\subset X$ be the exceptional divisor of $\phi_i$ 
and let $W_i\subset Y_i$ be the (smooth) center of the blowing up $\phi_i$. 

$(E\cdot R')=0$ and $(E\cdot R_i)>0$ holds since $(E\cdot R)<0$. 
Thus we have $(E\cdot R_1)>0$. 
Hence $R_1$ is also of type $(n-1,n-2)^{\sm}$ by Proposition \ref{step1}. 
Let $\phi_1:=\cont_{R_1}\colon X\rightarrow Y_1$, 
let $F_1\subset X$ be the exceptional divisor of $\phi_1$ 
and let $W_1\subset Y_1$ be the (smooth) center of the blowing up $\phi_1$. 

We note that $(F_i\cdot R)>0$ by Claim \ref{step311}. Thus $(F_i\cdot R_1)<0$ holds 
since $(F_i\cdot R)>0$, $(F_i\cdot R')=0$ and $(F_i\cdot R_i)<0$. 
Hence we have $C_1\subset F_i$ 
for any nontrivial fiber $C_1\subset X$ of $\phi_1$. 
Hence $F_1=F_i$ holds. 
We have $\phi_i(C_1)\neq\text{pt}$ and $Y_i$ is a Fano manifold 
by Proposition \ref{step2}.Thus we have 
\[
0<(-K_{Y_i}\cdot {\phi_i}_*C_1)=(-K_X\cdot C_1)+(F_i\cdot C_1).
\]
However, $(F_i\cdot C_1)=-1$ holds since $F_1=F_i$, and $(-K_X\cdot C_1)=1$. 
Thus $(-K_{Y_i}\cdot {\phi_i}_*C_1)=0$ holds. 
This leads to a contradiction. 
\end{proof}

We have $m=2$ and $(E\cdot R_2)=0$ and $(E\cdot R_1)>0$ by Claim \ref{step312}. 
The rays $R_1$ and $R_2$ are of type $(n-1,n-2)^{\sm}$ by Remark \ref{rmkcase1}. 
Therefore the ray $R_2$ is exactly the extremal ray what we want. 
\end{proof}

\begin{proposition}\label{step32}
If all extremal rays different from $R$ are of type $(n-1,n-2)^{\sm}$, 
there exists an extremal ray $R_0\subset\NE(X)$ with $(E\cdot R_0)=0$. 
\end{proposition}

\begin{proof}
We denote 
\[
\NE(X)=R+R_1+\cdots+R_m,
\]
where the set of $2$-dimensional extremal faces of $\NE(X)$ is $\{R+R_1, R_1+R_2,\cdots, R_{m-1}+R_m, R_m+R\}$. 
Let $\phi_i:=\cont_{R_i}\colon X\rightarrow Y_i$, 
let $F_i\subset X$ be the exceptional divisor of $\phi_i$ 
and let $W_i\subset Y_i$ be the (smooth) center of the blowing up $\phi_i$ 
for any $1\leq i\leq m$. 
We can assume that $(E\cdot R_1)>0$. 

Let $E_1:=\phi_1(E)\subset Y_1$. $Y_1$ is a Fano manifold 
having $\pr^1$-bundle structure By Proposition \ref{step2}. 
We consider the surjective map ${\phi_1}_*:\NE(X)\twoheadrightarrow\NE(Y_1)$. 
Let $R_{Y_1}$, ${R_2}_{Y_1}\subset\NE(Y_1)$ be the images of $R$, $R_2\subset\NE(X)$, respectively. 
Then the cone $\NE(Y_1)$ is spanned by $R_{Y_1}$, ${R_2}_{Y_1}$ 
and $\cont_{R_{Y_1}}(E_1)=\text{pt}$. 
Thus we have $Y_1\simeq\pr_Z(\sO_Z\oplus\sO_Z(s))\xrightarrow{\pi}Z$ such that 
$Z$ is a Fano manifold of $\rho_Z=1$ with index $r$ such that 
the ample generator of $\Pic(Z)$ is $\sO_Z(1)$ and $r>s\geq 0$ 
and $E_1$ is a section of $\pi$ (we note that $W_1\subset E_1$) by Proposition \ref{p1}. 
We consider the elementally transform of $\pi$ factoring $X$. 
More precisely, we consider the diagram 
\[\xymatrix{
& X \ar[dl]_{\phi'} \ar[dr]^{\phi_1} & \\
Y'\ar[dr]_{\pi'} & & Y_1 \ar[dl]^{\pi} \\
& Z & \\
}\]
with $Y'$ is a smooth projective variety, and $\pi'$ is a $\pr^1$-bundle 
and $\phi'$ is the blowing up (different from $\phi_0$) 
along a smooth subvariety of codimension $2$ 
such that the exceptional divisor is the strict transform of $\pi^{-1}\pi(W_1)$. 

There exists an extremal ray $R'\subset\NE(X)$ such that $(E\cdot R')=0$ 
since $\Exc(\phi')\cap E=\emptyset$. 
The ray $R'$ is exactly the extremal ray what we want. 
\end{proof}

As a consequence, we have completed the proof of Theorem \ref{mainthm}.

\smallskip

\noindent K.\ Fujita

Research Institute for Mathematical Sciences (RIMS),
Kyoto University, 

Oiwake-cho, Kitashirakawa, Sakyo-ku, Kyoto 606-8502, Japan 

fujita@kurims.kyoto-u.ac.jp


\medskip


\begin{thebibliography}{99}

\bibitem[And85]{ando}
T.\ Ando, \emph{On extremal rays of the higher-dimensional varieties}, 
Invent.\ Math.\ \textbf{81} (1985), no.\ 2, 347--357. 

\bibitem[BCW02]{bcw}
L.\ Bonavero, F.\ Campana, and J.\ A.\ Wi{\'s}niewski, \emph{Vari{\'e}t{\'e}s
  projectives complexes dont l'{\'e}clat{\'e}e en un point est de Fano},
  C.\ R.\ Math.\ Acad.\ Sci.\ Paris \textbf{334} (2002), no.\ 6, 463--468.

\bibitem[Cas09]{casagrande}
C.\ Casagrande, \emph{On Fano manifolds with a birational contraction sending a
  divisor to a curve}, Michigan Math.\ J.\ \textbf{58} (2009), no.\ 3, 783--805.

\bibitem[Fuj12]{fuj}
K.\ Fujita, \emph{The sum of the length minus one of extremal rays 
on Fano manifolds}, preprint.

\bibitem[KMM92]{KMM}
J.\ Koll{\'a}r, Y.\ Miyaoka and S.\ Mori, \emph{Rational connectedness and 
boundedness of Fano manifolds}, J.\ Differential Geom.\ \textbf{36} (1992), no.\ 3, 765--779. 

\bibitem[KM98]{KoMo}
J.\ Koll{\'a}r and S.\ Mori, \emph{Birational geometry of algebraic varieties}, 
Cambridge Tracts in Math, vol.\ 134, 
Cambridge University Press, Cambridge, 1998. 

\bibitem[MM81]{MoMu}
S.\ Mori and S.\ Mukai, \emph{Classification of Fano {$3$}-folds with 
$B_2\geq 2$}, 
Manuscripta Math.\ \textbf{36} (1981), no.\ 2, 147--162. 
Erratum: {\bf 110} (2003), no.\ 3, 407.

\bibitem[Tsu06]{tsukioka}
T.\ Tsukioka, \emph{Classification of Fano manifolds containing a negative divisor 
isomorphic to projective space}, Geom.\ Dedicata \textbf{123} (2006), 179--186.

\bibitem[Wi{\'s}91]{wisn}
J.\ A.\ Wi{\'s}niewski, \emph{On contractions of extremal rays of Fano manifolds}, J.\ Reine Angew.\ Math.\ 
\textbf{417} (1991), 141--157.

\end{thebibliography}
\end{document}